\documentclass[11pt,a4paper]{amsart}

\usepackage{amsthm, amsfonts, amssymb, amsmath, latexsym, enumerate,array}
\usepackage[all]{xy}
\usepackage{stmaryrd}
\usepackage[latin1]{inputenc}
\usepackage{hyperref,cite,mdwlist,mathrsfs}
\usepackage{tikz}
\usepackage[bitstream-charter]{mathdesign}

\usepackage{anysize}
\marginsize{3.3cm}{3.3cm}{3.3cm}{3.3cm}
\usepackage{setspace}

\newtheorem{THM}{Theorem}
\newtheorem{thm}{Theorem}[section]
\newtheorem{lem}[thm]{Lemma}
\newtheorem{corol}[thm]{Corollary}
\newtheorem{prop}[thm]{Proposition}
\newtheorem{claim}[thm]{Claim}

\newtheorem*{thm*}{Theorem}
\newtheorem*{cnj*}{Conjecture}

\theoremstyle{definition}
\newtheorem{rmk}[thm]{Remark}
\newtheorem{eg}[thm]{Example}
\newtheorem{dfn}[thm]{Definition}

\newtheorem*{conj*}{Conjecture}

\newcommand{\cA}{\mathcal{A}}

\newcommand{\cT}{\mathcal{T}}

\newcommand{\cS}{\mathcal{S}}

\newcommand{\cI}{\mathcal{I}}

\newcommand{\cO}{\mathcal{O}}

\DeclareMathOperator{\Ext}{Ext}
\DeclareMathOperator{\Hom}{Hom}
\DeclareMathOperator{\im}{Im}

\DeclareMathOperator{\HH}{H}

\newcommand{\rR}{\boldsymbol{R}}

\newcommand{\Z}{\mathbb Z}
\newcommand{\C}{\mathbb C}
\newcommand{\F}{\mathbb F}

\newcommand{\R}{\mathbb R}

\newcommand{\PP}{\mathbb P}
\newcommand{\PD}{\check{\mathbb P}}

\DeclareMathOperator{\ts}{\otimes}
\newcommand{\mono}{\hookrightarrow}
\newcommand{\epi}{\twoheadrightarrow}
\newcommand{\xr}{\xrightarrow}
\newcommand{\kk}{{\boldsymbol{k}}}
\newcommand{\TlogZ}{\cT_{\PP^n}(-\log D_Z)}
\newcommand{\TlogA}{\cT_{\PP^n}(-\log D_\cA)}
\newcommand{\TZ}{\cT_{Z}}
\newcommand{\mult}{\mathrm{mult}}

\allowdisplaybreaks[1]
\begin{document}


\title[Logarithmic bundles and line arrangements]{Logarithmic bundles
  and Line arrangements, \\
an
  approach via the standard construction}

\author{Daniele Faenzi}
\email{\tt daniele.faenzi@univ-pau.fr}
\address{Université de Pau et des Pays de l'Adour \\
  Avenue de l'Université - BP 576 - 64012 PAU Cedex - France}
\urladdr{{\url{http://univ-pau.fr/~faenzi/}}}

\author{Jean Vallès}
\email{{\tt jean.valles@univ-pau.fr}}
\address{Université de Pau et des Pays de l'Adour \\
  Avenue de l'Université - BP 576 - 64012 PAU Cedex - France}
\urladdr{\url{http://web.univ-pau.fr/~jvalles/jean.html}}



\thanks{Both authors partially supported by ANR-09-JCJC-0097-0
  INTERLOW and ANR GEOLMI}

\begin{abstract}
We propose an approach to study logarithmic sheaves $\cT_{\PP^n}(-\log D_\cA)$ associated with
hyperplane arrangements $\cA$ on the projective space $\PP^n$, based on projective
duality, direct image functors and vector bundles methods.
We focus on freeness of line arrangements having a point with high
multiplicity.
\end{abstract}

\maketitle

\section*{Introduction}

Let $\kk$ be a field of characteristic zero, and let
$\cA=(H_1,\ldots,H_m)$ be a hyperplane arrangement in
$\PP^n=\PP^n_\kk$, namely the $H_i$'s are distinct hyperplanes of $\PP^n$. 
The module of logarithmic derivations along the hyperplane arrangement
divisor $D_\cA = H_1 \cup \cdots \cup H_m$, and its sheaf-theoretic
counterpart $\TlogA$ (Saito's sheaf of logarithmic
vector fields) play a prominent role in
the study of $\cA$; let us only mention \cite{terao,schenck:bundle}.

One main issue in the theory of arrangements is to what extent the
sheaf $\TlogA$ depends on the combinatorial type of $\cA$, defined as 
the isomorphism type of the lattice $L_\cA$ of intersections of
hyperplanes in $\cA$.
This lattice is partially ordered by reverse
inclusion, and is equipped with a rank function given by codimension (cf. \cite{orlik-terao:arrangements}).
An important conjecture of Terao (reported in \cite{orlik-terao:arrangements}) asserts that if $\cA$ and $\cA'$ have
the same combinatorial type, and $\TlogA$ splits as a direct sum of line
bundles (i.e. $\cA$ is {\it free}), the same should happen to $\cT_{\PP^n}(-\log D_{\cA'})$.

In this paper we study the sheaf $\TlogA$ relating it to the finite
collection $Z$ of points in the dual space $\PD^n$ associated with
$\cA$ (we write $\cA=\cA_Z$ when $Z=\{z_1,\ldots,z_m\}$ satisfies
$H_i=H_{z_i}$ for all $i$, where $H_z \subset \PP^n$ denotes the hyperplane corresponding to a
point $z \in \PD^n$).
Our first result is that $\cT_{\PP^n}(-\log D_{\cA_Z})$ is obtained 
via the so-called standard construction from the ideal sheaf
$\cI_Z(1)$, namely it is the direct image of the ideal sheaf $\cI_Z(1)$ under the natural
correspondence between $\PP^n$ and $\PD^n$ (Theorem \ref{origine}).

On the projective plane, this allows to obtain special derivations (sections of
$\cT_{\PP^2}(-\log D_{\cA_Z})$) from points of high multiplicity in
$\cA_Z$. 
Using this observation, we show that an arrangement $\cA_Z$ of $2k+r+1$
lines with a point of multiplicity between $k$ and $k+r+1$ is free
with exponents $(k,k+r)$
if and only if $c_2(\cT_{\PP^2}(-\log D_{\cA_Z}))=k(k+r)$, see Theorem \ref{concurrent}.
Here, by definition, $\cA_Z$ free with exponents $(k,k+r)$ means that
$\cT_{\PP^2}(-\log D_{\cA_Z}) \simeq \cO_{\PP^2}(-k) \oplus
\cO_{\PP^2}(-r-k)$), and we write Chern classes on $\PP^n$ as
 integers, with obvious meaning.
Note that the second Chern class is a very weak invariant of the combinatorial type of $\cA_Z$.
For real arrangements, one can push this criterion to points of
slightly lower multiplicity, namely $k-1$ (in fact, a suitable
technical assumption is needed, see Theorem \ref{usa ungar} for the precise statement).

Next, we use the blow-up of the dual plane $\PD^2$ to show that the splitting of $\cT_{\PP^2}(-\log D_{\cA_Z})$
on a general line of $\PP^2$ is determined by the number $d_Z$,
defined as the minimal integer $d$ such that there exists a 
curve of degree $d+1$ in the dual plane $\PD^2$ passing through $Z$ and having
multiplicity $d$ at a general point of $\PD^2$ (Theorem \ref{generic}).

Using this approach, and the relations between the behavior of the
arrangement obtained removing a line $H_y$
from $\cA_Z$ and the possible order of trisecant lines to $Z$ in $\PD^2$
passing through $y$,
we are lead to show that freeness is a
combinatorial property for up to $12$ lines (Theorem \ref{12}).

\medskip

The paper is organised as follows.
In the next section we set up the main correspondence between ideal
sheaves of points in $\PD^n$ and the sheaf of logarithmic derivations on $\PP^n$.
Section \ref{high} contains our result on line arrangements 
having a point of high multiplicity.
In Section \ref{blow up} we show how to relate the number $d_Z$ and
the generic splitting of the sheaf of logarithmic derivations of $\cA_Z$.
In Section \ref{ungar} we develop the above mentioned refinement for
real arrangements.
In Section \ref{Deletion} we outline the relation of our method with
the technique of deletion of one line from an arrangement, with a
focus on freeness.
Finally, Section \ref{up to 12} is devoted to arrangements of $12$ or less lines in $\PP^2$.

\section{Duality and logarithmic vector fields}

Consider $\PP^n=\PP^n_\kk$, and let $Z=\{z_1,\ldots,z_m\}$ be a finite collection of points in the
dual space $\PD^n$.
Each point $y \in \PD^n$ corresponds to a hyperplane $H_{y}$ in
$\PP^n$ (and likewise we associate with $x\in \PP^n$ a hyperplane
of $\PD^n$, denoted by $L_x$). 
So with $Z$ we can associate the hyperplane arrangement $\cA_Z=(H_{z_1},\ldots,H_{z_m})$.
The hyperplane arrangement divisor $D_Z=D_{\cA_Z}$ is defined as $D_Z
= \cup_{i=1}^m H_{z_i}$. Let $f_i$ be a linear form defining $H_{z_i}$
and $f= \Pi_{i=1}^m f_i$ be an equation of $D_Z$.

Saito's sheaf of logarithmic vector fields $\TlogZ$ (see
\cite{saito:logarithmic}) is the sheafification of the module of
logarithmic derivations associated to $D_Z$, mod out by the Euler
derivation.
It can be obtained as
kernel of the map $\psi = (\partial_0 f, \ldots,\partial_n f)$, where
$x_0,\ldots,x_n$ are coordinates in $\PP^n$, and we write
$\partial_i=\partial/\partial(x_i)$.
We will often abbreviate $\TZ = \TlogZ$.

Our first result shows how to obtain $\TZ$ from the ideal
sheaf $\cI_{Z}$ of $Z$ in $\PD^n$ (we denote by $\cI_{X/Y}$ the
ideal sheaf of a subscheme $X$ of a scheme $Y$, and we suppress the 
notation $/Y$ when it is clear from the context).
Consider the flag variety:
\[
\F = \{(x,y) \in \PP^n \times \PD^n \, | \, x \in H_y\},
\]
and the projections $p$ and $q$ of $\F$ onto $\PP^n$ and $\PD^n$.
It is well-known that $\F \simeq \PP(\cT_{\PP^n}(-1))$.

\begin{THM} \label{origine}
  There is a natural isomorphism of sheaves of $\cO_{\PP^n}$-modules:
  \[
  \TZ \simeq p_*(q^*(\cI_Z(1))).
  \]
\end{THM}

\begin{proof} 
This is somehow implicit in \cite{faenzi-matei-valles:torelli,faenzi-valles:jaca}), still we give here a full proof.
Let us consider the canonical exact sequence of coherent
  sheaves of $\cO_{\PD^n}$-modules:
\[
0 \to \cI_Z(1) \to \cO_{\PD^n}(1) \to \cO_Z(1) \to 0.
\]
Applying $p_* \circ q^*$ to this sequence, we get a long exact sequence:
\begin{equation}
  \label{4 terms}
0 \to p_*(q^*(\cI_Z(1))) \to \cT_{\PP^n}(-1) \to p_*(q^*(\cO_Z(1))) \to \rR^1p_*(q^*(\cI_Z(1))) \to 0.
\end{equation}
\begin{claim} \label{R1}
The sheaf $\rR^1p_*(q^*(\cI_Z(1)))$ is supported at the points $x\in \PP^n$ such that $Z \cap L_x$ is not in general linear position.
\end{claim}
To see this, first note that:
\[
\rR^i p_*(q^*(\cI_Z(1)))=0, \mbox{for $i>1$}, 
\]
which follows from the easy fact that $\HH^i(L_x,\cI_{Z \cap
  L_x}(1))=0$ for all $i>1$ and all $x \in \PP^n$.
Then, by base change the support of the sheaf
$\rR^1p_*(q^*(\cI_Z(1)))$ is given by the points 
$x \in \PP^n$ such that 
$\HH^1(L_x,\cI_{Z \cap  L_x}(1)) \ne 0$, and this is non-zero if and only if
$Z\cap L_x$ is not in general linear position.

Next, we observe that, for any $t \in \Z$, there is a natural isomorphisms:
\begin{equation}
  \label{OZ}
p_*(q^*(\cO_Z(t))) \simeq \bigoplus_{z\in Z} \cO_{H_z}.  
\end{equation}

To see this, first recall that $\cO_Z \simeq \cO_Z(t)$ for all
$t$ since $Z$ has finite length.
Further, $p_*(q^*(\cO_Z(t)))$ can be seen simply as
$p_*(\cO_{q^{-1}(Z)})$ and since ${q^{-1}(Z)}$ is the disjoint union
of the $\{H_z \mid x \in Z\}$, we get the desired isomorphism.
\medskip

Let us now continue the proof of our theorem.
Let $\cS_Z$ be the singular locus of $D_Z$. This is defined by the
vanishing of all partial derivatives of $f$, i.e. the generators of $\cI_{\cS_Z/\PP^n}$
are given by the map $\psi$.
Since $\cS_Z \subset D_Z$, we have the natural exact sequence:
\[
0 \to \cO_{\PP^n}(-m) \xr{f} \cI_{\cS_Z/\PP^n} \to \cI_{\cS_Z/D_Z} \to 0.
\]
In view of the Euler relation, this sequence fits into a commutative diagram:
\begin{equation}
  \label{diagram}
\xymatrix@-2ex{
&& \cO_{\PP^n}(-1) \ar@{=}[r] \ar[d] & \cO_{\PP^n}(-1) \ar^{f}[d] \\
0\ar[r] & \TZ \ar@{=}[d] \ar[r] & \cO_{\PP^n}^{n+1}
\ar^-{\psi}[r] \ar[d] & \cI_{\cS_Z/\PP^n}(m-1) \ar[d] \ar[r] & 0 \\
0\ar[r] & \TZ   \ar[r] & \cT_{\PP^n}(-1) \ar[r] & \cI_{\cS_Z/D_Z}(m-1) \ar[r] & 0, 
}
\end{equation}
where the central column is the Euler sequence.
From \cite[Proposition 2.4]{dolgachev:logarithmic} desingularization
gives an inclusion of $\cI_{\cS_Z/D_Z}(m)$ into $p_*(\omega_{q^{-1} Z}
\ts \omega^*_{\PP^n}) \simeq p_*(q^*(\cO_Z))(1)$.

So $\TZ$ and $p_*(q^*(\cI_Z(1)))$ are both defined as kernel of maps
$\cT_{\PP^n}(-1) \to p_*(q^*(\cO_Z))$, and we want to see that this turns
them into isomorphic sheaves.

\begin{claim}
Any two sheaves defined as kernel of maps $\cT_{\PP^n}(-1) \to
\bigoplus_{z\in Z} \cO_{H_{z}}$ are isomorphic, provided that the
first Chern class of both of them is $1-m$.
\end{claim}
To see this, let $\alpha,\beta \in
\Hom_{\cO_{\PP^n}}(\cT_{\PP^n}(-1),\bigoplus_{z\in Z} \cO_{H_{z}})$
set $E = \ker(\alpha)$, $F=\ker(\beta)$, and assume $c_1(E)=c_1(F)=1-m$.
For all $z \in Z$, we have $\Hom_{\cO_{\PP^n}}(\cT_{\PP^n}(-1),\cO_{H_{z}}) \simeq
\HH^0(\PP^n,\Omega_{\PP^n}(1) \otimes \cO_{H_{z}})$.
Moreover,
there is a natural isomorphism $\Omega_{\PP^n}(1) \otimes
\cO_{H_z} \simeq \Omega_{H_z}(1) \oplus \cO_{H_z}$,
So we have:
\[
\mbox{$\Hom_{\cO_{\PP^n}}(\cT_{\PP^n}(-1),\bigoplus_{z\in Z} \cO_{H_{z}})
\simeq \bigoplus_{z\in Z} \HH^0(H_z,\cO_{H_z})$},
\]
indeed $\HH^0(H_y,\Omega_{H_y}(1))=0$.
Therefore, we may write $\alpha$ and $\beta$ as
$\alpha=(\alpha_z)_{z\in Z}$ and $\beta=(\beta_z)_{z\in Z}$ with
$\alpha_z,\beta_z \in \kk$.
The assumption $c_1(E)=1-m$ implies that $\alpha_z \ne 0$ for all $z\in
Z$, and analogously $\beta_z \ne 0$ for all $z\in Z$ since $c_1(F)=1-m$.
Now consider the automorphism $\gamma$ of $\bigoplus_{z\in Z} \cO_{H_{z}}$
defined on each factor $\cO_{H_z}$ as multiplication by $\frac{\beta_z}{\alpha_z}$.
We have a commutative diagram:
\[
\xymatrix{
  E \ar[r] \ar@{.>}^{\delta}[d] & \cT_{\PP^n}(-1) \ar^-{\alpha}[r] \ar@{=}[d] & \bigoplus_{z\in Z} \cO_{H_{z}} \ar^{\gamma}[d]\\ 
  F \ar[r] & \cT_{\PP^n}(-1) \ar^-{\beta}[r] & \bigoplus_{z\in Z} \cO_{H_{z}}
}
\]
The map $\delta : E \to F$ induced by $\gamma$ has an inverse induced
by $\gamma^{-1}$, and we get that $E$ is isomorphic to $F$.

To finish the proof, we only need to recall that $p_*(q^*(\cO_Z))
\simeq \bigoplus_{z\in Z} \cO_{H_{z}}$ by \eqref{OZ}, and 
the first Chern class of both $\TZ$ and
$p_*(q^*(\cI_Z(1)))$ equals $1-m$: for $\TZ$ if follows from instance
from \eqref{diagram}, and for $p_*(q^*(\cI_Z(1)))$ it follows from
\eqref{4 terms}, since $\rR^1p_*(q^*(\cI_Z(1)))$ is supported in
codimension at least $2$ by Claim \ref{R1}.
This finishes the proof of the theorem.
\end{proof}

As an example of application of this description of $\TZ$ as direct
image, let us mention the following well-known result (a quick proof
will be given in the next section).

\begin{prop} \label{facile}
  Let $\cA$ be an arrangement of $m$ lines, $k \ge 0$ be an integer, $x$ be a
  point of multiplicity $k+1$ of $D_{\cA}$. Set $\cA'=\cA \setminus \{H
    \in \cA \mid x \in H\}$.
  Then the following are equivalent:
  \begin{enumerate}[i)]
  \item \label{un} the arrangement $\cA$ is free with exponents $(k,m-k-1)$;
  \item \label{deux} any point of multiplicity $h \ge 2$ in $D_{\cA'}$ has multiplicity
    $h+1$ in $D_\cA$.
  \end{enumerate}
\end{prop}


\section{Line arrangements with a point of high multiplicity}
\label{high}
Here we study freeness of line arrangement that admit a point having high
multiplicity with respect to the exponents.
Recall that a line arrangement $\cA_Z$ is {\it free with exponents
  $(a,b)$} if $\TZ \simeq \cO_{\PP^2}(-a) \oplus \cO_{\PP^2}(-b)$.
Of course, this implies that $c_2(\TZ)=a b$.

\begin{THM} \label{concurrent}
  Let $k \ge 1$, $r\ge 0$ be integers, set $m=2k+r+1$, and consider
  a line arrangement $\cA$ of $m$ lines with a point of
  multiplicity $h$ with $k \le h \le k+r+1$.
  Then $\cA$ is free with exponents $(k,k+r)$ if and only if
  $c_2(\cT_{\PP^2}(-\log D_\cA))=k(k+r)$.
\end{THM}

\begin{rmk}
  In the above setting, it turns out that if $h\ge k+r+2$, then $\cA$
  cannot be free with exponents $(k,k+r)$, see Corollary \ref{res to line}.
\end{rmk}

\begin{rmk}
A transparent way to compute the Chern class $c_2(\cT_{\PP^2}(-\log D_\cA))$ is the following.
Set $b_{\cA,h}$ for the number of points of multiplicity $h$ of
$D_\cA$ (we will also call them the points of multiplicity $h$ ``of $\cA$'').
Then we have the relations:
  \begin{align}
    \label{multipli} b_{\cA,2}+3b_{\cA,3}+6b_{\cA,4} +\cdots & = {m \choose 2}, \\
    \label{multipli chern} b_{\cA,3}+3b_{\cA,4}+6b_{\cA,5} +\cdots & ={m-1 \choose 2} - c_2(\TlogA).
  \end{align}
The second equality, valid when not all lines in $\cA$ pass through
the same point, follows immediately from the long exact sequence:
\begin{equation}
  \label{steiner}
0 \to \cT_Z \to \cO_{\PP^n}^{m-1} \to \cO_{\PP^n}^{m-3}(-1) \to
\rR^1p_*(q^*(\cI_{Z}(1))) \to 0,
\end{equation}
obtained by resolving $q^*(\cI_Z(1))$ in $\PP^2 \times \PD^2$ via:
\[
0 \to pr_1^*(\cO_{\PP^2}(-1)) \otimes pr_2^*(\cI_Z(1)) \to
pr_2^*(\cI_Z(1)) \to q^*(\cI_Z(1)) \to 0.
\]
Here, $pr_i$ denote the projections onto the factors of $\PP^2 \times
\PD^2$ (we refer to \cite{faenzi-matei-valles:torelli} for more on the
matrix of linear forms appearing in \eqref{steiner}).
Indeed, in view of the proof of Claim \ref{R1}, the sheaf
$\rR^1p_*(q^*(\cI_{Z}(1)))$ is the direct sum of
the $\cO_{\langle x_j^{m_j-2}\rangle}$, 
where the $x_j$'s vary in
the (set-theoretic) support of $\cS_Z$ and $m_j=
\mult(D_{\cA_Z},x_j)$.
Here $\langle x^i\rangle$ denotes the $(i-1)^{\mathrm{th}}$
infinitesimal neighborhood of $x$ in $\PP^2$.
This is the subscheme cut
in $\PP^2$ by the $i^{\mathrm{th}}$ power of the ideal defining $x$.
It has length ${i+1 \choose 2}$.
Formula \eqref{multipli chern} thus follows
computing Chern classes in \eqref{steiner}.
\end{rmk}

\begin{eg}
  The result gives a quick way to show that an arrangement having
  the combinatorial type of the Hesse 
  arrangement of the $12$ lines
  passing through the $9$ inflection points of a smooth complex
  plane cubic is free with exponents $(4,7)$.
\begin{figure}[h!]
  \centering
\begin{tikzpicture}[scale=1.5]
\draw (0,0) node {$\bullet$};
\draw (-1,1) node {$\bullet$};
\draw (-1,-1) node {$\bullet$};
\draw (1,-1) node {$\bullet$};
\draw (1,1) node {$\bullet$};
\draw (-1,0.5) node {$\bullet$};
\draw (-0.5,-1) node {$\bullet$};
\draw (1,-0.5) node {$\bullet$};
\draw (0.5,1) node {$\bullet$};
\draw (-2,1) -- (2,1);
\draw (-2,-1) -- (2,-1);
\draw (-2,-1) -- (2,-1);
\draw (-1,-1.5) -- (-1,1.5);
\draw (1,-1.5) -- (1,1.5);
\draw (-1.5,-1.5) -- (1.5,1.5);
\draw (-1.5,1.5) -- (1.5,-1.5);
\draw (-0.75,-1.5) -- (0.75,1.5);
\draw (-1.5,0.75) -- (1.5,-0.75);
\draw (-7/6,1) -- (-0.5,-1);
\draw (-1,-7/6) -- (1,-0.5);
\draw (7/6,-1) -- (0.5,1);
\draw (1,7/6) -- (-1,0.5);
\draw [dashed] (-1,1) arc (90:270:13/12);
\draw [dashed] (1,1) arc (0:180:13/12);
\draw [dashed] (1,-1) arc (-90:90:13/12);
\draw [dashed] (-1,-1) arc (180:360:13/12);
\end{tikzpicture}
  \caption{Hesse arrangement}
\end{figure}
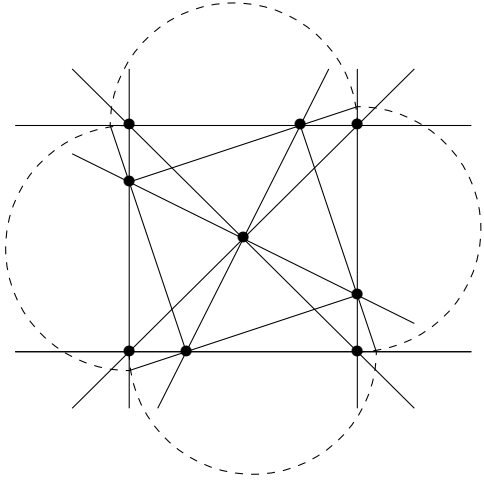
\end{eg}

To prove the theorem, we will need the following lemma.
The way to use it will frequently be by contradiction: assume that a 
bundle $E$ with Chern classes as below does not split, hence take a non-zero element of $\HH^0(\PP^2,E(-1))$,
and look for a contradiction with some other property.
We will sometimes call a non-zero element of $\HH^0(\PP^2,E(-1))$
an {\it unstable section}.

\begin{lem} \label{splitting}
  Let $E$ be a rank-$2$ vector bundle on $\PP^2$ and assume
  $c_1(E)=-r$ for some $r\ge 0$ and $c_2(E)=0$.
  Then, the following are equivalent:
  \begin{enumerate}[i)]
  \item \label{splits} the bundle $E$ splits as $\cO_{\PP^2} \oplus \cO_{\PP^2}(-r)$,
  \item \label{no sections} we have $\HH^0(\PP^2,E(-1))=0$,
  \item \label{one line} there is a line $H$ of $\PP^2$ such that $E_{|H} \simeq \cO_H
    \oplus \cO_H(-r)$.
  \end{enumerate}
  For any line $H$ of $\PP^2$ we have $E_{|H} \simeq \cO_H(s)
    \oplus \cO_H(-r-s)$, for some integer $s\ge 0$.
\end{lem}

\begin{proof}
  Condition \eqref{splits} clearly implies \eqref{no sections}.
  The equivalence of \eqref{splits} and
  \eqref{one line} is proved in \cite{elencwajg-forster}.
  So it only remains to show that \eqref{no sections} implies
  \eqref{splits}, which we will now do.
  
  Let $t$ be the smallest integer such that $\HH^0(\PP^2,E(t)) \ne 0$.
  By \eqref{no sections} we know $t \ge 0$.
  Also, it is well-known that any non-zero global section $s$ of $E(t)$
  vanishes along a subscheme $W$ of $\PP^2$ of codimension $\ge 2$ and
  of length:
  \begin{equation}
    \label{c2}
  c_2(E(t))=t(t-r) \ge 0.
  \end{equation}
  We have an exact sequence:
  \[
  0 \to \cO_{\PP^2} \xr{s} E(t) \to \cI_W(2t-r) \to 0.
  \]
  So $t=0$ would imply $X = \emptyset$ hence $\cI_W(2t-r) \simeq
  \cO_{\PP^2}(-r)$ and $E$ splits as $\cO_{\PP^2} \oplus \cO_{\PP^2}(-r)$ since
  $\Ext^1_{\PP^2}(\cO_{\PP^2}(-r),\cO_{\PP^2})=0$.

  Then, it remains to rule out the case $t>0$.
  Hence, we assume $t>0$ i.e. $\HH^0(\PP^2,E)=0$, and we look for a contradiction.
  By Riemann-Roch, the Euler characteristic $\chi(E)$ is positive,
  hence $\HH^2(\PP^2,E) \ne 0$, so $\HH^0(\PP^2,E(r-3)) \ne 0$ by
  Serre duality, indeed $E^* \simeq E(r)$. 
  Therefore $t>0$ implies $t \le r-3$.
  But by \eqref{c2}, $t>0$ implies $t \ge r$, a contradiction.

  Let us now prove the last statement.
  Given a line $H$ of $\PP^2$, we have $E_{|H} \simeq \cO_H(s)
  \oplus \cO_H(-r-s)$ for some integer $s$, and we have to check
  that $s$ is non-negative.
  Let us assume $s<0$, and show that this leads to a contradiction.
  First, note that we may assume $s>-r$, for otherwise posing
  $s'=-r-s$ we have $s' \ge 0$ and we still have 
  $E_{|H} \simeq \cO_H(s')
  \oplus \cO_H(-r-s')$.
 
  Now, in case $-r<s<0$, we have an unstable section, namely $\HH^0(\PP^2,E(-1)) \ne 0$ since $E$ does
  not decompose as $\cO_{\PP^2} \oplus \cO_{\PP^2}(-r)$ (by the part
  we have already proved of this lemma).
  For all integers $t$, the exact sequence of restriction of $E(t)$ to
  $H$ reads:
  \[
  0 \to E(t-1) \to E(t) \to \cO_H(t+s) \oplus \cO_H(t-r-s) \to 0.
  \]
  So $-r<s<0$ implies $\HH^0(\PP^2,E(t-1)) \simeq \HH^0(\PP^2,E(t))$ for all $t \le
  0$, and this space is zero for $t \ll 0$.
  But this contradicts $\HH^0(\PP^2,E(-1)) \ne 0$.
\end{proof}

We will now prove our theorem.
We call a
line $L\subset \PD^2$ a {\it $h$-secant line to $Z$} if $|L \cap Z|
\ge h$.
We add the adjective {\it strict} if we require equality.
Recall that for all $h$, the number
$b_{\cA_Z,h}$ is the number of strict $h$-secants to $Z$.

\begin{proof}[Proof of Theorem \ref{concurrent}]
  One direction if obvious.
  What we have to prove is that the condition on Chern classes is
  sufficient, so we assume $c_2(\cT_{\PP^2}(-\log D_\cA))=k(k+r)$.
  Let $Z$ be the set of $m$ points of $\PD^2$ corresponding to
  $\cA$, so that $\cA=\cA_Z$ and $\cT_{\PP^2}(-\log D_\cA)=\TZ$.
  Since $\cA$ has a point $x$ of multiplicity $h\ge k$,
  on the dual side there is a line $L = L_x\subset \PD^2$ that contains $h$
  points of $Z$ (i.e. $L$ is a strict $h$-secant to $Z$), and leaves
  out the remaining $m-h$ points of 
  $Z$.
  Set $Z'= Z \setminus L$. Let $g$ be an equation of $L$ in $\PD^2$.

  Restricting the ideal sheaf $\cI_Z$ to $L$ we get the ideal
  sheaf of $h$ points in $\PP^1$, i.e. $\cO_L(-h)$. This gives an exact sequence:
  \begin{equation}
    \label{reduction}
    0 \to \cI_{Z'} \xr{g} \cI_Z(1) \to \cO_L(1-h) \to 0. 
  \end{equation}
  We apply $p_* \circ q^*$ to this exact sequence.
  It is easy to see that $p_*(q^*(\cI_{Z'}))$ only depends on the
  length of $Z'$ and is isomorphic to $\cO_{\PP^2}(h-m)$.
  Similarly, it is not hard to check (where in the second formula
  we replace the RHS by zero for $h \le 2$):
  \begin{align*}
    & p_*(q^*(\cO_L(1-h))) \simeq \cO_{\PP^2}(1-h), \\  
    & \rR^1p_*(q^*(\cO_L(1-h))) \simeq \cO_{\langle x^{h-2}\rangle},
  \end{align*}
  (recall that $\langle x^i\rangle$ denotes the $(i-1)^{\mathrm{th}}$
  infinitesimal neighborhood of $x$ in $\PP^2$).

  Therefore $p_* \circ q^*$ of \eqref{reduction} gives:
  \begin{align}
  \label{long} 0 \to & \cO_{\PP^2}(h-m) \to \TZ \xr{\alpha} \cO_{\PP^2}(1-h) \to \\   
  \nonumber \to & \rR^1p_*(q^*(\cI_{Z'})) \to\rR^1p_*(q^*(\cI_{Z}(1))) \to 
  \cO_{\langle x^{h-2}\rangle} \to 0.
 \end{align}
  Now, an argument similar to Claim \ref{R1} shows that
  $\rR^1p_*(q^*(\cI_{Z'}))$ is supported at points $x$ such that 
  $\HH^1(L_x,\cI_{Z' \cap  L_x}) \ne 0$, i.e. such that $|L_x \cap Z'|
  \ge 2$ (bisecant lines to $Z'$).
 The image of the map $\alpha$ above is then a sub-sheaf of
 $\cO_{\PP^2}(1-h)$, whose first Chern class is $1-h$ since all the
 sheaves in the second row of \eqref{long} are supported in
 codimension $\ge 2$.
 This means that $\im(\alpha) \simeq \cI_\Gamma(1-h)$, for some finite
 length subscheme $\Gamma \subset \PP^2$, and we have:
 \begin{equation}
   \label{Gamma}
   0 \to \cO_{\PP^2}(h-m) \to \TZ \to \cI_\Gamma(1-h)\to 0.
 \end{equation}
 This subscheme parametrizes bisecant lines to $Z'$ that meet $L$ away
 from $Z$.

 We apply now Lemma \ref{splitting}. If, by contradiction, the bundle
 $\TZ \otimes \cO_{\PP^2}(k)$ did not split as $\cO_{\PP^2} \oplus
 \cO_{\PP^2}(-r)$, then we would have an unstable section, namely:
 \[
 \HH^0(\PP^2,\TZ \otimes \cO_{\PP^2}(k-1)) \ne 0.
 \]

 Note that the assumption $h \le k+r+1=m-k$ gives $h+k-m-1<0$, so we have
  the vanishing $\HH^0(\PP^2,\cO_{\PP^2}(h+k-m-1))=0$.
 So, from \eqref{Gamma}, twisted by
 $\cO_{\PP^2}(k-1)$, we deduce:
 \[
 \HH^0(\PP^2, \cI_\Gamma(k-h)) \ne 0,
 \]
 hence clearly $k\ge h$, which implies $h=k$ so $\HH^0(\PP^2,
 \cI_\Gamma) \ne 0$. This says that $\Gamma$ is empty.
 But computing Chern classes in \eqref{Gamma} twisted by
 $\cO_{\PP^2}(k-1)$ (and still with $h=k$) shows
 that $\Gamma$ has length $c_2(\cI_\Gamma)=r+1$, a contradiction.
\end{proof}

\begin{proof}[Proof of Proposition \ref{facile}]
Again, we let $Z$ be the set of $m$ points of $\PD^2$ corresponding to
$\cA$, so that $\cA=\cA_Z$ and $\cT_{\PP^2}(-\log D_\cA)=\TZ$.
  Since $\cA$ has a point $x$ of multiplicity $k+1$,
  on the dual side there is a line $L = L_x\subset \PD^2$ that contains $k+1$
  points of $Z$ (i.e. $L$ is a strict $k+1$-secant to $Z$), and leaves
  out the remaining $m-k-1$ points of 
  $Z$.
  Set $Z'= Z \setminus L$. 
  We have $\cA'=\cA_{Z'}$.
  We can then rewrite \eqref{long} as:
  \begin{align}
  \label{long-2} 0 \to & \cO_{\PP^2}(k-m+1) \to \TZ \to \cO_{\PP^2}(-k) \to \\
   \nonumber \to & \rR^1p_*(q^*(\cI_{Z'})) \to\rR^1p_*(q^*(\cI_{Z}(1))) \to 
  \cO_{\langle x^{k-1}\rangle} \to 0.
  \end{align}
  
  Now, \eqref{deux} is equivalent to the fact that, for any $h \ge 2$, any strict
  $h$-secant to $Z'$ is $(h+1)$-secant to $Z$.
  By the interpretation we gave in previous proof of the higher direct
  images appearing in \eqref{long-2}, this is equivalent to exactness
  of the sequence:
  \[ 0 \to \rR^1p_*(q^*(\cI_{Z'})) \to\rR^1p_*(q^*(\cI_{Z}(1))) \to
  \cO_{\langle x^{k-1}\rangle} \to 0.\]
  By \eqref{long-2}, this is equivalent to exactness of:
  \[
  0 \to \cO_{\PP^2}(k-m+1) \to \TZ \to \cO_{\PP^2}(-k) \to 0.
  \]
  This is clearly equivalent to \eqref{un}.
\end{proof}

\section{Blowing up of the dual plane and restriction to lines}
\label{blow up}

Let $H=H_y$ be a line in $\PP^2$.
Given a finite set of $m$ points $Z$ in $\PD^n$, the associated
sheaf $\TZ$ restricts to $H_y$ as: 
\[
(\TZ)_{|H_y} \simeq \cO_{H_y}(-a_y) \oplus \cO_{H_y}(-b_y),
\]
for some integers $a_y\le b_y$ with $a_y+b_y=m-1$.
Let us work out the dual picture.

\begin{dfn}
Let $y \in \PD^2$ and let $Z$ be a finite set of points of $\PD^2$.
We define $d_{Z,y}$ as the smallest positive integer $d$ such that
{\it there is a curve in $\PD^2$ of degree $d+1$ passing through $Z$ and
  having multiplicity $d$ at $y$}.
Equivalently, $d_{Z,y}$ is the smallest integer $d$ such that:
\[
\HH^0(\PD^2,\cI^d_{y} \ts \cI_Z(d+1)) \ne 0.
\]
We also define $d_Z$ as $\max_{y \in \PD^2} d_{Z,y}$.
\end{dfn}

\begin{THM} \label{generic}
  Let $Z$ be a finite set of points of $\PD^2$ and $y\in \PD^2
  \setminus Z$. If $y$ lies on no trisecant line to $Z$, then $a_y=d_{Z,y}$.
\end{THM}

\begin{proof}
  We will first outline the application to our situation of the so-called {\it standard construction}, see \cite{okonek-schneider-spindler}.
  We consider the blow-up $\tilde{\PP}$ of $\PD^2$ at the point $y$,
  and we recall that $\tilde{\PP}=p^{-1}(H_y)$, where $p$ is the
  projection map from the flag $\F$ to $\PP^2$. 
  Therefore we have an exact sequence:
  \[
  0 \to p^*(\cO_{\PP^2}(-1)) \to \cO_\F \to \cO_{\tilde{\PP}} \to 0,
  \]
  with $\cO_{\tilde{\PP}} \simeq p^*(\cO_{H_y})$.
  We denote by $\tilde p$ and $\tilde q$ the induced projections
  from $\tilde \PP$ to $H_y$ and to $\PD^2$.
  Tensoring the above exact sequence by $q^*(\cI_Z(1))$ and taking
  direct image by $p$ we get the long exact sequence:
  \begin{align*}
  0 \to & \cT_Z(-1) \xr{f_y} \cT_Z \to \tilde p_* (\tilde q^*(\cI_Z(1))) \to \\
  \to & \rR^1p_*(q^*(\cI_{Z}(1))) \xr{f_y} \rR^1p_*(q^*(\cI_{Z}(1))) \to 
  \rR^1\tilde p_*(\tilde q^*(\cI_{Z}(1))) \to 0,
  \end{align*}
  where here $f_y$ is an equation of $H_y$ in $\PP^2$.
  Note that, in the second row of the above diagram, all sheaves have
  finite length by Claim \ref{R1}, and the kernel of the instance of
  $f_y$ in this row has the same length as its cokernel.
  Further note that, by base change over $x \in H_y$, we have:
  \[
  \rR^1\tilde p_*(\tilde q^*(\cI_{Z}(1))) \ts \kk_x \simeq
  \HH^1(L_x,\cI_{Z \cap L_x}(1)),
  \]
  where $\kk_x$ is the residue field at $x$.
  Therefore, looking at the line $L_x$ in $\PD^2$ (which passes
  through $y$),
  the number of points in the intersection $L_x \cap Z$  equals 
  $\dim(\rR^1\tilde p_*(\tilde q^*(\cI_{Z}(1))) \ts \kk_x)+2$
  (it is understood that if $|L_x \cap Z| \le 2$ then this vector space
  is zero).
  Summing up, we have an exact commutative diagram:
\begin{center}
\begin{tikzpicture}[scale=2.0]
\node (0primo) at (-4,0) {$0$} ;
\node (TZmeno) at (-3,0) {$\TZ(-1)$} ;
\node (TZ) at (-2,0) {$\TZ$} ;
\node (tilde) at (-0.4,0) {$\tilde p_* (\tilde q^*(\cI_Z(1)))$} ;
\node (tau) at (1,0) {$\boldsymbol{\tau}_{Z,y}$};
\node (0secondo) at (2,0) {$0$} ;
\node (destro) at (-0.5,-1) {$0$} ;
\node (sinistro) at (-2,-1) {$0$} ;
\node (immagine) at (-1.2,-0.5) {$(\cT_Z)_{|H_y}$} ;
\draw[->,>=stealth] (0primo) to (TZmeno);
\draw[->,>=stealth] (TZmeno) -- (TZ) node[midway,above] {$f_y$};
\draw[->,>=stealth] (TZ) to (tilde);
\draw[->,>=stealth] (tilde) to (tau);
\draw[->,>=stealth] (tau) to (0secondo);
\draw[->,>=stealth] (TZ) to (immagine);
\draw[->,>=stealth] (sinistro) to (immagine);
\draw[->,>=stealth] (immagine) to (tilde);
\draw[->,>=stealth] (immagine) to (destro);
\end{tikzpicture}
\end{center}
  Here, the sheaf $\boldsymbol{\tau}_{Z,y}$ has length $t_{Z,y}$.
  In particular, if $t_{Z,y} = 0$ then:
  \[
  \tilde p_* (\tilde q^*(\cI_Z(1))) \simeq (\cT_Z)_{|H_y}.
  \]

  Let us now head to the proof of our theorem. Set $a=a_y$, $b=b_y$.
  The decomposition $(\TZ)_{|H_y} \simeq \cO_{H_y}(-a) \oplus
  \cO_{H_y}(-b)$ gives an injective map $\cO_{H_y}(-a) \to \tilde p_* (\tilde q^*(\cI_Z(1)))$.
  Pulling back to $\tilde \PP$, since $\tilde p^*(\cO_{H_y}) \simeq
  \cO_{\tilde \PP}$, we get an injection:
  \[
   \cO_{\tilde \PP} \mono \tilde
  q^*(\cI_Z(1)) \ts \tilde p^*(\cO_{H_y}(a)).
  \]
  Pushing down to $\PD^2$, we get a map:
  \[
   \cO_{\PD^2} \mono \cI_Z(1) \ts \tilde q_*(\tilde p^*(\cO_{H_y}(a))) \simeq
   \cI_Z(1) \ts \cI^a_y(a).
  \]
  
  Therefore we get $d_{Z,y} \ge a$.
  To check the opposite inequality, we show
  $\HH^0(\PD^2,\cI^{a-1}_{y} \ts \cI_Z(a)) = 0$.
  A non-zero element of this space, once pulled-back to $\PD^2$, would give a 
  map $\cO_{\tilde \PP^2} \mono \cI_Z(a)$, vanishing with multiplicity
  $a-1$ along the exceptional divisor $\tilde q^{-1}(y)$.
  By clearing $a-1$ times the equation of $\tilde q^{-1}(y)$, we get a map:
  \[
   \cO_{\tilde \PP} \mono \tilde
  q^*(\cI_Z(1)) \ts \tilde p^*(\cO_{H_y}(a-1)),
  \]
  hence, by pushing forward to $H_y$ via $\tilde p$, a map:
  \[
  \cO_{H_y}(1-a) \mono \tilde p_* (\tilde q^*(\cI_Z(1))).
  \]
  This is incompatible with the splitting $(\TZ)_{|H_y} \simeq \cO_{H_y}(-a) \oplus
  \cO_{H_y}(-b)$, with $a\le b$.
\end{proof}

Putting together the previous theorem and Lemma \ref{splitting},
we get the following result, somehow related to Yoshinaga's
theorem, cf. \cite{yoshinaga:free}.

\begin{corol} \label{res to line}
  Let $k \ge 1$, $r\ge 0$ be integers, set $m=2k+r+1$, and consider
  a line arrangement $\cA_Z$ associated with $m$ points $Z$ in $\PD^2$
  having $c_2(\TZ)=k(k+r)$.
  Then the following are equivalent:
  \begin{enumerate}[i)]
  \item the arrangement $\cA_Z$ is free with exponents $(k,k+r)$;
  \item there is a line $H=H_y$ in $\PP^2$ such that $(\TZ)_{|H_y} \simeq \cO_{H_y}(-k) \oplus \cO_{H_y}(-k-r)$; 
  \item for a point $y \in \PD^2$ lying in no trisecant line to $Z$, we have $d_{Z,y}=k$;
  \item we have $d_Z = k$.
  \end{enumerate}
  In particular, if $Z$ has a $h$-secant line with $h \ge k+r+2$, then
  $\cA_Z$ cannot be free.
\end{corol}

\begin{proof}
The only two observations needed for the proof of the equivalence of
the first four statements, besides Lemma
\ref{splitting} and Theorem \ref{generic}, is
the fact the maximal value
$d_Z = k$ among the $d_{Z,y}$'s is attained at a generic point $y \in \PD^2$.
This in turn is clear by upper semicontinuity of the function $y \mapsto \dim_{\kk}
\HH^0(\PD^2,\cI_{y}^d \ts \cI_Z(d+1))$.

For the last statement we argue as follows.
Suppose that $Z$ has a strict $h$-secant line $L$ with $h \ge k+r+2$
and take a general point $y$ in $\PD^2$.
Then we have a curve in $\PD^2$ of degree $m-h+1 \le k$ through $Z$ and having
multiplicity $m-h$ at $y$, namely the union of $L$ and of the $m-h$
lines joining $y$ with the $m-h$ points of $Z \setminus L$.
Therefore $d_{Z,y} \le k-1$ so $\cA_Z$ is not free by the previous statements of this corollary.
\end{proof}
\section{Two double points aligned with
  a point of high multiplicity} 

\label{ungar}

For this section only, we consider {\it real} arrangements, namely we
let $\kk=\R$.
Let $k \ge 1$, $r\ge 0$ be integers, and set $m=2k+r+1$.
We consider here line arrangements $\cA$ of $m$ lines having:
\begin{itemize}
\item a point $x_0$ of multiplicity $k-1$.
\item two points $x_1$, $x_2$, each of multiplicity at least $2$ in
  $\cA$, such that $\{x_0,x_1,x_2\}$ is contained in a line $H$ which
  is not in $\cA$.
\end{itemize}

\begin{THM} \label{usa ungar}
  Let $\cA$ be as above. Then $\cA$ is free with exponents $(k,k+r)$
  if and only if $c_2(\cT_{\PP^2}(-\log D_\cA))=k(k+r)$.
\end{THM}

\begin{proof}
  Let $Z$ be the set of points in $\PP^2$ corresponding to $\cA$, so $\cA =
  \cA_Z$.
  We take a $(k-1)$-tuple point $x_0$ of 
  $\cA$, and let $H$ be the line in $\PP^2$ containing
  $x_0,x_1,x_2$, with $x_1,x_2$ of multiplicity $2$ or higher in $\cA$.
  Again we let $L=L_x$ be the $(k-1)$-secant to $Z$ in $\PD^2$, and
  set $Z' = Z \setminus L$, this time $Z'$ consists of $k+r+2$ points.
  We consider the exact sequence \eqref{reduction} (where this time
  $h=k-1$) and the
  long exact obtained by applying $p_* \circ q^*$ to it:
  \begin{align*}
   0 \to & \cO_{\PP^2}(-k-r-2) \to \TZ \xr{\alpha} \cO_{\PP^2}(2-k) \to \\   
  \nonumber \to & \rR^1p_*(q^*(\cI_{Z'})) \to\rR^1p_*(q^*(\cI_{Z}(1))) \to 
  \cO_{\langle x^{k-3}\rangle} \to 0.
 \end{align*}
 
 This time, the image of the map $\alpha$ above is $\cI_\Gamma(2-k)$,
 where $\Gamma$ has length $2r+2$, and the above sequence becomes:
 \begin{equation}
   \label{riscrivo con 3}
 0 \to \cO_{\PP^2}(-k-r-2) \to \TZ \to \cI_{\Gamma}(2-k) \to 0.   
 \end{equation}

 Let us assume now that $\TZ$ is not free, and show that this leads to
 a contradiction.
 First, by Lemma \ref{splitting} we
 have $\HH^0(\PP^2,\TZ(k-1)) \ne 0$, and we see from \eqref{riscrivo con
   3} that this gives $\HH^0(\PP^2,\cI_\Gamma(1)) \ne 0$, therefore $\Gamma$ is
 contained in a line $H$ of $\PP^2$.
 According to the interpretation  we gave in the proof of
 Theorem \ref{concurrent}, the subscheme $\Gamma$ parametrizes (with
 multiplicity) the set of bisecant lines to $Z'$ that meet $L$ away from $Z$.
 Since $x_1$ and $x_2$ give two such bisecants, we have that $H=H_w$.
Note that by assumption $w$ does not lie in $Z$.

Summing up, we have the set $Z'$ of $k+r+2$ points in $\PD^2$, and a
set $Z''$ of $k$ points in $L$, constituted by $w$ and $Z\cap L$, such that
any bisecant line to $Z'$ cuts $L$ along $Z''$.
If we let now $L$ be the line at infinity in $\PD^2$, we see that $Z'$ is a
set of $k+r+2$ points of an affine $2$-dimensional space,
that determines at most $k$ directions.
But, since we are working over $\R$, the set $Z'$ should determine at
least $k+r+1 \ge k+1$ directions, according to Ungar's theorem, see
\cite{ungar:2N}.
This is a contradiction.
\end{proof}

\section{Sub-arrangement obtained by deletion}

\label{Deletion}

A classical and useful technique in the theory of arrangements consists in
considering arrangements obtained from an arrangement $\cA$ by adding
a hyperplane out of $\cA$, or deleting one of $\cA$, or restricting
$\cA$ to a hyperplane of $\cA$ (see \cite{orlik-terao:arrangements}
for a comprehensive treatment).
Here we provide a different approach to this technique and outline
some considerations on freeness of line arrangements based on our approach.
Most of the results contained in this section are certainly known to
experts, and can be proved with the classical techniques of deletion.

\subsection{Deletion of a point and triple points along a hyperplane}

Let $Z$ be a finite set of points in $\PD^n$ and let $z\in Z$.
Set $Z' = Z \setminus \{z\}$.
We say that $\cA_{Z'}$ is a sub-arrangement of $\cA_Z$, obtained by
{\it deletion} of $z$.
We have the exact sequence:
\[
0 \to \cI_{Z} \to \cI_{Z'} \to \cO_z \to 0. 
\]
Applying $p_* \circ q^*$ to this sequence, we get:
  \begin{align*}
   0 \to \TZ \to \cT_{Z'} \xr{\beta_0} \cO_{H_z} \xr{\beta_1} \rR^1p_*(q^*(\cI_{Z}(1))) \xr{\beta_2}\rR^1p_*(q^*(\cI_{Z'}(1))) \to 0.
 \end{align*}

The kernel of the map $\beta_1$ above is a sub-sheaf of $\cO_{H_z}$,
which we refer to as the ideal of {\it triple points of $\cA_Z$ along $H_z$}.
For a point $y \in \PD^2$, we write:
\[
t_{Z,y} = \sum_{x \in H_y \cap \cS_Z} (\mult(D_{\cA_Z},x) -2).
\]

\begin{prop}
  We have a short exact sequence:
  \begin{equation}
    \label{deletion}
  0 \to \TZ \to \cT_{Z'} \to \cO_{H_z}(-t_{Z,z}) \to 0.    
  \end{equation}
\end{prop}

\begin{proof}
  Given a point $x$ in $\PP^2$, 
  we denote again by $\langle x^i\rangle$ the $(i-1)^{\mathrm{th}}$
  infinitesimal neighborhood of $x$ in $\PP^2$.

  We have said that the sheaf
  $\rR^1p_*(q^*(\cI_{Z}(1)))$ is the direct sum of
  the $\cO_{\langle x_j^{m_j-2}\rangle}$, where the $x_j$'s vary in
  the support of $\cS_Z$ and $m_j= \mult(D_{\cA_Z},x_j)$.
  Therefore, the kernel of the map $\beta_2$ above 
  describes the difference between triple points of $\cA_Z$ and 
  triple points of $\cA_{Z'}$ each counted with multiplicity.
  By computing multiplicities, we get that the length of the support
  of $\ker(\beta_2)$ is precisely $t_{Z,z}$.
  Since $\ker(\beta_2) = \im(\beta_1)$ has length $t_{Z,z}$, we get
  that $\ker(\beta_1)=\im(\beta_0)$ has degree $-t_{Z,z}$, hence it is
  just $\cO_{H_z}(-t_{Z,z})$.
\end{proof}

\subsection{Some properties of freeness of line arrangements related
  to deletion}

Here we give some simple relations between freeness of a given
arrangements $\cA_Z$ and the numbers $t_{Z,z}$, for $z \in Z$.
Throughout the subsection, we let $k\ge 1$, $r\ge 0$ be integers, set
$m=2k+r+1$, and we consider a set $Z$ of $m$ points of $\PP^2$ and the
corresponding line arrangement $\cA_Z$.

\begin{prop}
  Assume $\cA_Z$ is free with exponents $(k,k+r)$, let $z \in Z$ and
  set $Z'=Z\setminus \{z\}$.
  Then, one of the 
  following alternatives takes place:
  \begin{enumerate}[i)]
  \item \label{minimo} $t_{Z,z}=k-1$ and $\cA_{Z'}$ is free with exponents $(k-1,k+r)$;
  \item \label{massimo} $t_{Z,z}=k+r-1$ and $\cA_{Z'}$ is free with exponents $(k,k+r-1)$;
  \item \label{tertio}
    $t_{Z,z} \ge k+r$ and $\cA_{Z'}$ is not
     free.
  \end{enumerate}
\end{prop}

\begin{proof}
  Dualizing the exact sequence \eqref{deletion} (i.e., applying to it
  the functor $\mathcal{H}om_{\cO_{\PP^2}}(-,{\cO_{\PP^2}})$), 
  using the fact that
  $\mathcal{E}xt^1_{\cO_{\PP^2}}(\cO_{H_Z}(-t),{\cO_{\PP^2}}) \simeq
  \cO_{H_z}(t+1)$ for all integer $t$, we obtain an exact sequence:
  \begin{equation} \label{duale}
  0\to \cT_{Z'}^*\to \cT_{Z}^* \to \cO_{H_z}(t_{Z,z}+1)\to 0.
  \end{equation}
  Here $(-)^*$ denotes the dual of a vector bundle.
  Since $\cT_{Z}^* \simeq \cO_{\PP^2}(k) \oplus \cO_{\PP^2}(k+r)$, we have thus a
  a surjective map:
  \[
  \cO_{\PP^2}(k) \oplus \cO_{\PP^2}(k+r) \epi \cO_{H_z}(t_{Z,z}+1).
  \]

  Then, it
  is  clear that $t_{Z,z} \ge k -1$ for otherwise there could not be an
  epimorphism as above. 
  Also, it is clear that in case \eqref{minimo} 
  the kernel bundle of the above map splits in the desired way, since
  the map above factors as:
  \[
  \cT_{Z}^* \to \cO_{\PP^2}(k) \epi \cO_{H_z}(k),
  \]
  where the first map is the projection onto the direct summand
  $\cO_{\PP^2}(k)$ and the second map is the canonical surjection.
  The case \eqref{massimo} is analogous.

  Let us prove now the case \eqref{tertio}. We consider again the exact sequence \eqref{duale}.
      We twist it by $-t_{Z,z}-1$ and take the long exact sequence of cohomology. Since $t_{Z,z} \ge k+r$ we get 
  $\HH^1(\PP^2,\cT_{Z'}^*(-t_{Z,z}-1))\neq 0$ which proves that $\cT_{Z'}$ does not decompose as a direct sum of line bundles.
\end{proof}

In the same spirit, we have the following proposition.

\begin{prop} \label{tZ}
Assume  $c_2(\TZ)=k(k+r)$. Then:
\begin{enumerate}[i)]
\item \label{non in mezzo} for all $z \in Z$, we have $t_{Z,z} \not\in ]k-1,k+r-1[$;
\item \label{uguale} if there is $z \in Z$ such that $t_{Z,z} = k-1$ or $t_{Z,z} =
  k+r-1$, then $\cA_Z$ is free with exponents $(k,k+r)$; 
\item \label{minore} if there is $z \in Z$ such that $t_{Z,z} < k-1$,
 then $\cA_Z$ is not free.
\end{enumerate}
Moreover, if $\cA_Z$ is not free, but has the same combinatorial type
of a free arrangement, then for all $z \in Z$ we have $t_{Z,z} \ge k+r$.
\end{prop}

\begin{proof}
  Consider again the exact sequence obtained in the proof of the
  previous proposition (from which we borrow the notation also):
  \[
  0 \to \cT_{Z'}^* \to \cT_{Z}^* \to \cO_{H_z}(t_{Z,z}+1) \to 0.
  \]
  Consider now the restriction to the line $H_z$ of $\cT_{Z}^*$.
  This splits as $\cO_{H_z}(k-s) \oplus \cO_{H_z}(k+r+s)$, for some
  integer $s \ge 0$ by Lemma \ref{splitting}, indeed one computes $c_2(\TZ(-k))=0$.
  So we get an epimorphism:
  \begin{equation}
    \label{epi-Hz}
  \cO_{H_z}(k-s) \oplus \cO_{H_z}(k+r+s) \epi \cO_{H_z}(t_{Z,z}+1).    
  \end{equation}
  Now, in case $t_{Z,z} = k-1$ or $t_{Z,z} = k+r-1$, this forces
  $s=0$, hence $\TZ$ is free by Corollary \ref{res to line}.
  This gives \eqref{uguale}.
  By the same corollary, since $t_{Z,z} < k-1$ forces $s<0$, we get
   \eqref{minore}.
  To see \eqref{non in mezzo},
  we note that an epimorphism of the form \eqref{epi-Hz} cannot exist in this range.

  To check the last statement, note that $\cA_Z$ cannot have the
  combinatorial type of a free arrangement $\cA_{Z_0}$ if
  $t_{Z,z}<k-1$, for necessarily we have
  $c_2(\cT_{Z_0})=k(k+r)$ and we would get $t_{Z_0,z_0}<k-1$ for some
  $z_0 \in Z_0$
  contradicting \eqref{minore}.
  Also, we cannot have $t_{Z,z}=k-1$ or $t_{Z,z}=k+r-1$ for any $z\in
  Z$ for otherwise $\cA_Z$ would be free by \eqref{uguale}.
  Then by \eqref{non in mezzo} we get $t_{Z,z} \ge k+r$ for all $z
  \in Z$.
\end{proof}

\section{Combinatorial nature of freeness for up to $12$ lines}

\label{up to 12}


In this section we work over $\kk=\C$.
The aim here will be to show that Terao's conjecture holds for
arrangements of up to $12$ complex projective lines.
As far as we
know, this had been checked for up to $10$ lines, see \cite{wakefield-yuzvinsky}.
The reason for devoting a paragraph to such a little progress is that
we still hope that our methods involving
the {\it unstable section} of Section \ref{high} can help treating more
cases.
Also, it seems to us the combinatorial subtleties are ruled
out by Theorem \ref{concurrent} for less than $11$ lines, so this
seems to be the first intriguing case.

\begin{THM} \label{12}
  Terao's conjecture holds for up to $12$ lines in $\PP^2_\C$.
\end{THM}

We fix again our notation: we consider a finite set of point $Z$ in
$\PD^2$ and the corresponding arrangement $\cA = \cA_Z$.

\begin{lem}
  Assume $\cA$ is free with exponents $(k,k+r)$, with $r \ge 0$.
  \begin{enumerate}[i)]
  \item \label{no higher} There is no $h$-secant line to $Z$, for $h > k+r+1$.
  \item \label{classify 4} Assume that $Z$ is non-degenerate and has no $4$-secant line.
  Then $(k,k+r)$ take value $(1, 1)$, $(1, 2)$, $(2, 2)$, $(2, 3)$, $(3, 3)$,
  $(3, 4)$ or $(4, 4)$.
  In the last case, $\cA$ has the combinatorial type of the 
  arrangement ({\it sometimes called ``dual Hesse}) given by the $9$
  lines corresponding to the inflection 
  points of a smooth cubic curve in $\PD^2$.
  \end{enumerate}
\end{lem}

\begin{proof}
  The first claim follows by looking at the exact sequence
  \eqref{reduction}.
  Indeed, let $L$ be a strict $h$-secant line to $Z$, with $h>k+r+1$.
  Then $Z'=Z \setminus L$ consists of $2k+r+1-h$ points, and we have
  an inclusion $\cI_{Z'} \to \cI_Z(1)$.
  Applying $p_* \circ q^*$ to this inclusion we get a nonzero map:
  \[
  \cO_{\PP^2}(h-2k-r-1) \to \TZ \simeq \cO_{\PP^2}(-k) \oplus \cO_{\PP^2}(-k-r).
  \]
  This is impossible since $h-2k-r-1>-k$, and we get \eqref{no higher}.

  Let us look at \eqref{classify 4}.
  Since $Z$ is non-degenerate, we have $k \ge 1$ and by \eqref{no
    higher} we get $k+r \le 4$.
  The relations \eqref{multipli} and \eqref{multipli chern} yield:
  \[
  b_{\cA,2}=-k^2-ks-s^2+2k-2s.
  \]
  It is easy to see that this gives a negative number for $r\ge 2$, or
  for $r=1$, $k\ge 4$, and also for $r=0$, $k\ge 5$.
  This leaves out the desired cases only.
  In case $r=0$, $k=4$, we get $b_{\cA,3}=12$ and $b_{\cA,3}=0$,
  which is the combinatorial type of the arrangement of $9$ lines dual
  to the $9$ inflection points of a smooth cubic curve.
\end{proof}

\begin{proof}[Proof of Theorem \ref{12}]
  The case when $Z$ is degenerate is trivial.
  Theorem \ref{concurrent} and the previous lemma only leave out the cases of
  line arrangements $\cA$ of $11$ lines with $c_2(\TlogA)=25$ or of
  $12$ lines with $c_2(\TlogA)=30$, having in both cases a quadruple
  point and no quintuple points.

  We look at the first case, which in our notation has $k=5$,
  $r=0$; the second case is completely analogous, and in fact a bit
  easier, so we will only say a word about it at the end of the proof.
  We consider thus an arrangement $\cA$ with $c_2(\TlogA)=25$ and such
  that $\TlogA$ is not free.
  Also, we can assume that there is a line arrangement $\cA_0$, with
  the same combinatorial type of $\cA$, such
  that $\cT_{\PP^2}(-\log D_{\cA_0})$ is free with exponents $(5,5)$.

  Let $Z$ and $Z_0$ be the corresponding set of points in $\PP^2$, so $\cA =
  \cA_Z$, $\cA_0=\cA_{Z_0}$, let $x$ be a
  quadruple point of $\cA$, and set $L$ for 
  the corresponding $4$-secant line to $Z$ in $\PD^2$.
  We are in the situation of Theorem \ref{usa ungar}, only without the
  assumptions of reality and of existence of $x_1,x_2$.
  However we let $Z' = Z \setminus L$ (this time $Z'$ consists of $7$ points),
  and we consider the exact sequence \eqref{reduction} and the
  long exact obtained by applying $p_* \circ q^*$:
  \begin{align*}
   0 \to & \cO_{\PP^2}(-7) \to \TZ \xr{\alpha} \cO_{\PP^2}(-3) 
  \nonumber \to \rR^1p_*(q^*(\cI_{Z'})) \to\rR^1p_*(q^*(\cI_{Z}(1))) \to 
  \cO_{\langle x^{2}\rangle} \to 0.
 \end{align*}
 Again $\im(\alpha)=\cI_\Gamma(-3)$, where 
 $\Gamma$ is a subscheme of $\PP^2$ of length $4$, so that the above sequence becomes:
 \begin{equation} \label{attenzione}
 0 \to \cO_{\PP^2}(-7) \to \TZ \to \cI_{\Gamma}(-3) \to 0.   
 \end{equation}
 Since $\TZ$ is not free by assumption, by Lemma \ref{splitting} we
 have $\HH^0(\PP^2,\TZ(4)) \ne 0$, and again we get
 $\HH^0(\PP^2,\cI_\Gamma(1)) \ne 0$ so that $\Gamma$ is
 contained in a line $H_w$, for some $w \in \PD^2$.

\begin{claim} \label{not in Z}
The point $w$ does not lie in $Z$.
\end{claim}

\begin{proof}[Proof of claim \ref{not in Z}]
Recall that, by the proof of
Theorem \ref{concurrent}, the subscheme $\Gamma$ parametrizes (with
multiplicity) the set of bisecant lines to $Z'$ that meet $L$ away from $Z$.
The fact that $\Gamma$ sits in $H_w$ means that these lines all meet at $w$.

If $w$ belongs to $Z$, then this is a combinatorial property that
must also hold for $Z_0$, namely the subscheme $\Gamma_0$
 associated to $Z_0$ should be contained in a line $H_{w_0}$
 corresponding to the meeting point $w_0$.
But $Z_0$ is free with exponents $(5,5)$, so by Lemma 
\ref{splitting} we
have
$\HH^0(\PP^2,\cT_{Z_0}(4)) = \HH^0(\PP^2,\cI_{\Gamma_0}(1))=0$.
 Hence $\Gamma_0$ lies in no line.
\end{proof}

\begin{claim} \label{in L}
The point $w$ lies in $L$, and the extended arrangement $\tilde Z = Z \cup \{w\}$
is free with exponents $(4,7)$.
\end{claim}

\begin{proof}[Proof of claim \ref{in L}]
Since $\Gamma$ has length $4$ and lies in $H_w$, we get a projection
$\cI_\Gamma \to \cO_{H_w}(-4)$, with kernel $\cO_{\PP^2}(-4)$, hence
an exact sequence:
\[
0 \to \cO_{\PP^2}(-4) \to \cI_\Gamma(-3) \to \cO_{H_w}(-7) \to 0.
\]
Composing the above projection with the epimorphism $\TZ \to
\cI_\Gamma(-3)$ we get a surjective map $\TZ \to \cO_{H_w}(-7)$.
This gives an exact sequence:
\[
0 \to \cO_{\PP^2}(-7) \oplus \cO_{\PP^2}(-4) \to \TZ \to \cO_{H_w}(-7) \to 0,
\]
obtained by observing that the map $\cO_{\PP^2}(-7) \to \TZ$ does not factor 
through $\cO_{\PP^2}(-4) \to \TZ$ (indeed, in that case its cokernel would have torsion).

Now, we would like to show that the map $\TZ \to \cO_{H_w}(-7)$ is the
restriction map associated according to \eqref{deletion} to deletion of $H_w$ from the arrangement
corresponding to $\tilde Z$,
so that its kernel $\cO_{\PP^2}(-7) \oplus \cO_{\PP^2}(-4)$ will be
$\cT_{\tilde Z}$, thus proving the claim on freeness of $\tilde Z$.
For this, it suffices to show that $t_{{\tilde Z},w}=7$.
We rewrite \eqref{deletion} as:
\[
0 \to \cT_{\tilde Z} \to \TZ \to \cO_{H_w}(-t_{{\tilde Z},w}) \to 0.
\]
Since we have a surjection $\TZ \epi \cO_{H_w}(-7)$, we must have
$(\TZ)_{|H_w} \cong \cO_{H_w}(-3) \oplus \cO_{H_w}(-7)$.
Therefore, by the exact sequence above we deduce $t_{{\tilde Z},w}=7$
or $t_{{\tilde Z},w} \le 3$.
But $t_{{\tilde Z},w} \ge 4$ since we have at least the $4$ bisecant
lines to $Z$ meeting at $w$ coming from $\Gamma$.
The claim is now proved.
\end{proof}

Let us now finish the proof of the theorem (still for $m=11$).
We have several cases, according to the configuration of $\Gamma$.
We give figures for the first two cases.

\begin{figure}[h]
  \centering
\begin{tikzpicture}[scale=0.8]
\draw (-0.5,0) node[above] {$w$}; \draw (0,0) node {$\bullet$};
\draw (1,0) node[below] {$z_8$}; \draw (1,0) node {$\bullet$};
\draw (2,0) node[below] {$z_9$}; \draw (2,0) node {$\bullet$};
\draw (3,0) node[below] {$z_{10}$}; \draw (3,0) node {$\bullet$};
\draw (4,0) node[below] {$z_{11}$}; \draw (4,0) node {$\bullet$};
\draw (1,1) node[right] {$z_1$}; \draw (1,1) node {$\bullet$};
\draw (2,2) node[right] {$z_2$}; \draw (2,2) node {$\bullet$};
\draw (3,3) node[below] {$z_3$}; \draw (3,3) node {$\bullet$};
\draw (0,1) node[left] {$z_5$}; \draw (0,1) node {$\bullet$};
\draw (0,2) node[left] {$z_6$}; \draw (0,2) node {$\bullet$};
\draw (0,3) node[left] {$z_7$}; \draw (0,3) node {$\bullet$};
\draw (3,1) node[left] {$z_4$}; \draw (3,1) node {$\bullet$};
\draw[very thick] (-1,0) -- (6,0) ;
\draw (-1,-1) -- (4,4);
\draw (0,-1) -- (0,4);
\draw[thick] (5.5,0) node[above] {$L$};
\end{tikzpicture}
  \caption{Case \eqref{2+2}}
\end{figure}
\begin{enumerate}[i)]
\item {\it $\Gamma$ consists of $2$ strict $3$-secants to $Z$ passing
  through $w$}. \label{2+2} In this case we have a point $z \in Z$ lying off the
lines of $\Gamma$ ($z=z_4$ in the figure above). Since $Z$ shares the combinatorial type with a free
arrangement, we have $t_{\tilde Z,z} \ge 5$ by
Proposition \ref{tZ}. Moreover, $\tilde Z$ is free with exponents
$(4,7)$, so in fact again Proposition \ref{tZ} yields $t_{\tilde Z,z} \ge 6$.
We see that this give two strict $4$-secant lines to $Z$ through $z$
(although they do not show up in the picture above!).
Finally we have $b_{\tilde Z,5}=1$ (the line $L$) and $b_{\tilde Z,4}=4$ (the pairs
of $4$-secants through $z$ and $w$).
This gives $b_{\tilde Z,3}=9$ by \eqref{multipli chern}, hence 
$b_{\tilde Z,2}=5$ by \eqref{multipli}.
But this contradicts {\it Hirzebruch's inequality} (see
\cite{hirzebruch:arrangements}), in the ``improved'' version:
\[
\mbox{$b_{Z,2} + \frac 34 b_{Z,3} \ge m + \sum_{h \ge 5} (2h-9)b_{Z,h}$},
\]
valid for arrangements $Z$ of $m$ complex projective lines having $b_{Z,m}=b_{Z,m-1}=b_{Z,m-2}=0$.

\begin{figure}[h]
  \centering
\begin{tikzpicture}[scale=0.8]
\draw (-0.5,0) node[above] {$w$}; \draw (0,0) node {$\bullet$};
\draw (1,0) node[below] {$z_8$}; \draw (1,0) node {$\bullet$};
\draw (2,0) node[below] {$z_9$}; \draw (2,0) node {$\bullet$};
\draw (3,0) node[below] {$z_{10}$}; \draw (3,0) node {$\bullet$};
\draw (4,0) node[below] {$z_{11}$}; \draw (4,0) node {$\bullet$};
\draw (1,1) node[right] {$z_1$}; \draw (1,1) node {$\bullet$};
\draw (2,2) node[right] {$z_2$}; \draw (2,2) node {$\bullet$};
\draw (2,1) node[below] {$z_3$}; \draw (2,1) node {$\bullet$};
\draw (4,2) node[below] {$z_4$}; \draw (4,2) node {$\bullet$};
\draw (0,1) node[left] {$z_5$}; \draw (0,1) node {$\bullet$};
\draw (0,2) node[left] {$z_6$}; \draw (0,2) node {$\bullet$};
\draw (0,3) node[left] {$z_7$}; \draw (0,3) node {$\bullet$};
\draw[very thick] (-1,0) -- (6,0) ;
\draw (-1,-1) -- (3,3);
\draw (-2,-1) -- (6,3);
\draw (0,-1) -- (0,4);
\draw[thick] (5.5,0) node[above] {$L$};
\draw (0,4) node[left] {$L_3$};
\draw (3,3) node[left] {$L_1$};
\draw (6,3) node[below] {$L_2$};
\end{tikzpicture}
  \caption{Case \eqref{1+1+2}}
\end{figure}
\item \label{1+1+2} {\it $\Gamma$ consists of $2$ strict bisecant lines $L_1,L_2$ and one strict
    $3$-secant line $L_3$ to $Z$ passing through $w$}.
  Let the $\{z_{2i-1},z_{2i}\} = L_i \cap Z$ for $i=1,2$.
  In this case, by the same argument as we see that, for all $i,j \in \{1,2\}$ the
  lines through $z_{i}$ and $z_{2+j}$ has to meet $L_3$ and $L$
  along a point of $Z$ (although this cannot be seen in the picture above!).
  We get $b_{\tilde Z,5}=1$ and $b_{\tilde Z,4}=5$.
  Then $b_{\tilde Z,3}=6$ by \eqref{multipli chern}, hence 
  $b_{\tilde Z,2}=8$ by \eqref{multipli}.
  Again this contradicts Hirzebruch's inequality.
\item {\it $\Gamma$ consists of one strict bisecant line and one strict
    $4$-secant line to $Z$ passing through $w$}.
  In this case we have $b_{\tilde Z,5}=2$ and $b_{\tilde Z,4}=0$.
  Using formulas \eqref{multipli chern} and \eqref{multipli} again we
  obtain a contradiction with Hirzebruch's inequality.
\item {\it $\Gamma$ consists of one strict $5$-secant line 
    to $Z$ passing through $w$}.
  This case is analogous to the previous one and we omit it.
\end{enumerate}

Finally, let us only sketch the idea for $m=12$.
We have to deal with an arrangement $\cA=\cA_Z$ given by $12$ points
$Z \subset \PD^2$
having a $4$-secant line $L$ and no $5$-secant lines.
Again we set $Z'=Z \cap L$ and $Z''=Z\setminus L$.
By the same argument as the for the case $m=11$, we will have a point
$w \in L \setminus Z$, such that all bisecant lines to $Z''$ meet $L$ either at
a point of $Z$, either at $w$.
We have an exact sequence analogous to \eqref{attenzione}, that
provides with, this time, $6$ such bisecants through $w$ (counted with
multiplicity).
Since $Z$ has no $5$-secant lines, this leaves as the only
possibility that there are $2$ strict $4$-secant lines to $Z''$
through $w$, in such a way that all other bisecant to $Z''$ meets $L$
along $Z'$.
This can be proved to be impossible again by Hirzebruch's inequality,
or also by a simple lemma due to Kelly, \cite[Lemma 2]{kelly:serre}.
\end{proof}

\bibliographystyle{alpha}
\bibliography{bibliography}
\end{document}